\documentclass[11pt]{amsart}
\usepackage{ams_style}
\usepackage{anyfontsize}
\usepackage{tikz-cd}
\usepackage[backend=biber, sorting=nyt, minbibnames=10, maxbibnames=10]{biblatex}
\usepackage{hyperref} \hypersetup{colorlinks, allcolors=black}
\title{Isolated Points on Modular Curves of Prime-Power Level}
\author{Chris Calger}
\address{Wake Forest University, Winston-Salem, NC 27104}
\email{calgcs24@wfu.edu}

\theoremstyle{plain}

\DeclareEmphSequence{\sffamily\bfseries}
\newcommand{\enumref}[2]{{\hyperref[#2]{\ref*{#1}.\ref*{#2}}}}%

\begin{document}

\begin{abstract}
    An isolated point on an algebraic curve is a closed point not belonging to a collection of points of the same degree parametrized by $\mathbf{P}^1$ or a positive rank abelian subvariety of the curve's Jacobian. 
    % We study the set of $j$-invariants in an extension of bounded degree that arise as the $j$-invariant of an isolated point on a modular curve.
    We study the sets of $j$-invariants, in extensions of bounded degree, that arise as the $j$-invariant of an isolated point on a modular curve. 
    % We obtain finiteness results on this set for families of modular curves with prime-power level. 
    We obtain finiteness results on these sets for families of modular curves with prime-power level. 
    This is related to recent work of Bourdon and Ejder, who classified rational $j$-invariants of isolated points on the families $X_1(n)$ and $X_0(n)$, for $n$ a prime power. 
\end{abstract}

\maketitle

\section{Introduction}

% \comment{What should be included in this section?}
% To include in the introduction:
% \begin{enumerate}
%     \item Motivate the problem
%     \begin{enumerate}
%         \item what is an isolated point? (short)
%         \begin{enumerate}
%             \item faltings thm
%         \end{enumerate}
%         \item what is a modular curve? (short)
%         \item (!) why study isolated points on modular curves?
%         \begin{enumerate}
%             \item serre's uniformity problem
%             \item mazur's program B
%             \item classifying $\ell$-adic images
%         \end{enumerate}
%     \end{enumerate}
%     \item Explain prior work
%     \item state theorem
%     \item techniques used/highlight key parts
%     \begin{enumerate}
%         \item Cadoret-tamagawa
%     \end{enumerate}
%     \item any additional prior work
%     \item other: code, acknowledgments, etc.
% \end{enumerate}

\par 
The closed points of a smooth, projective curve $C$ over a number field $k$ are in bijection with $\Gal_k$-orbits of points in $C(\overline{k})$, where $\Gal_k \defeq \Gal(\overline{k}/k)$ is the absolute Galois group. 
The degree of a closed point corresponds to the size of the associated Galois-orbit.
% A closed point on a smooth, projective curve $C$ over a number field $k$ is a $\Gal_k$-orbit of points in $C(\overline{k})$, where $\Gal_k \defeq \Gal(\overline{k}/k)$ is the absolute Galois group.
% The size of this orbit is called the degree of the associated closed point. 
% The degree of a closed point is the size of the associated Galois-orbit.
A closed point $x \in C$ of degree $d$ is said to be \emph{isolated}\footnote{See Section \ref{subsection:isolated points} for a precise definition.} if it does not belong to a collection of degree $d$ closed points parametrized by $\mathbf{P}^1$ or a positive rank abelian subvariety of the curve's Jacobian. 

\par 
Faltings's Theorem \cite{Fal83} guarantees that if $C$ has genus strictly greater than $1$, then $C$ has finitely many degree $1$ points.
% There are infinitely many points of $C$ of arbitrary degree, but it was shown in \cite{Bou19} that only finitely many of these points are isolated. \notice 
Such a curve $C$ has infinitely many points of arbitrary degree. 
However, it was shown in \cite{Bou19} that, as a consequence of Faltings's Theorem on rational points on subvarieties of abelian varieties \cite{Fal94},
$C$ has only finitely many isolated points.
% , varying over points of all degrees. 
Because of this, the study of isolated points on curves generalizes the study of rational points in a natural way. 

\par 
% For $\ell \geq 2$ a prime integer we denote the \emph{$\ell$-adic integers} by $\ZZ_\ell$. 
% If $\ell$ a prime integer and $E$ an elliptic curve defined over a number field $k$, the action of $\Gal_k$ on the $\ell$-adic Tate module $T_\ell(E) \defeq \projlim_n E[\ell^n]$ induces a representation
% \[ \rho_{E, \ell^\infty} : \Gal_k \to \Aut(T_\ell(E)) \simeq \GL_2(\ZZ_\ell) = \projlim_n \GL_2(\ZZ/\ell^n\ZZ). \]
% called the \emph{$\ell$-adic Galois representation associated to $E$}.
% If $\ell$ are integers and $E$ an elliptic curve defined over a number field $k$, the action of $\Gal_k$ on the $\ell$-torsion subgroup $E[\ell]$ and on the $\ell$-adic Tate module $T_\ell(E) \defeq \projlim_n E[\ell^n]$ induce representations
% \comment{transition sentence}
%
In this paper, we study isolated points on modular curves, which parametrize elliptic curves equipped with specific structure. 
If $E$ is an elliptic curve defined over a number field $k$, the action of $\Gal_k$ on the $n$-torsion subgroups $E[n]$, for positive integers $n$, induces representations
\begin{align*}
    \rho_{E, n} &: \Gal_k \to \GL_2(\ZZ/n\ZZ),
    \\ \rho_{E, \ell^\infty} &: \Gal_k \to \GL_2(\ZZ_\ell) = \projlim \GL_2(\ZZ/\ell^n\ZZ), \text{ and}
    \\ \rho_{E} &: \Gal_k \to \GL_2(\hat{\ZZ}) = \projlim \GL_2(\ZZ/n\ZZ).
\end{align*}
If $H$ is an open subgroup of $\GL_2(\ZZ_\ell)$, the modular curve $X_H$ is a curve whose non-cuspidal points parametrize elliptic curves with 
$\rho_{E, \ell^\infty}(\Gal_k)$ 
% $\ell$-adic image of Galois 
contained in $H$. To a closed point $x \in X_H$, we associate an algebraic number $j(x)$ called the \emph{$j$-invariant of $x$}, which is the image of $x$ under the natural map to the coarse moduli space of elliptic curves. 

\par
% \comment{\& Why study isolated points on modular curves!? (SUC, Mazur's program B, classifying $\ell$-adic images)}
Isolated points on modular curves are related to a number of open questions concerning the arithmetic of elliptic curves. 
For a fixed non-CM elliptic curve $E$ over a number field $k$, Serre showed \cite{Ser72} that the mod-$\ell$ Galois representation $\rho_{E, \ell}$ is surjective for large primes $\ell$. Serre asked \cite[Section 4.3]{Ser72} if this result in the case of $k=\QQ$ can be made uniform---that is, if there is prime $L$ such that $\rho_{E,\ell}$ is surjective for all non-CM elliptic curves $E/\QQ$ and all primes $\ell > L$. This prime $L$ has since been conjectured by Zywina \cite[Conjecture 1.2]{Zyw24} and Sutherland \cite[Conjecture 1.1]{Sut16} to be $37$. We will refer to this conjecture as \emph{Serre's uniformity problem}.
It remains unknown if there exists 
a prime $\ell > 37$ and non-CM elliptic curve $E/\QQ$ with $\rho_{E, \ell}(\Gal_k)$ conjugate to $\Cns(\ell)$, the normalizer of a non-split Cartan subgroup of $\GL_2(\ZZ/\ell\ZZ)$. 
% \par
% \comment{connect to isolated points}
If such a prime $\ell$ and elliptic curve $E$ exists, then $E$ gives rise to a representative for a degree one point on the modular curve $X_{\Cns(\ell)}$. This curve has genus greater than $1$, so Faltings's Theorem implies there are finitely many degree one points on $X_{\Cns(\ell)}$, all of which are isolated. 

\par 
% \comment{improve transition}
% More generally, the problem of classifying the possible images of Galois representations attached to elliptic curves is encapsulated by Mazur's
% The study of rational isolated points on modular curves is also related to 
%
The problem of bounding the non-surjective mod-$\ell$ images of Galois is situated within a broader effort, namely Mazur's ``Program B'' \cite{Maz77}, which asks, given a number field $k$, for a classification of the images $\rho_{E}(\Gal_k)$ in $\GL_2(\hat{\ZZ})$ for elliptic curves $E/k$. 
% \comment{$\ell$-adic images}
% In this paper we study modular curves 
% In this paper, we consider modular curves $X_H$, where $H$ is an open subgroup of $\GL_2(\ZZ_\ell)$, which is related to the $\ell$-adic case of 
Work in this paper is related to the $\ell$-adic case of the program, which is the problem of classifying images $\rho_{E,\ell^\infty}(\Gal_k)$ in $\GL_2(\ZZ_\ell)$. 
Determining the $\ell$-adic image of Galois attached to an elliptic curve $E/\QQ$ helps us to understand closed and isolated points with $j$-invariant $j(E)$ on modular curves, and vice-versa. 

\par
For each positive integer $n$, the modular curve $X_1(n)$ is a curve whose non-cuspidal points parametrize elliptic curves with a point of order $n$. Bourdon, Ejder, Liu, Odumodu, and Viray \cite[Corollary 1.7]{Bou19} showed that, assuming an affirmative answer to Serre's uniformity problem, there are finitely many $j$-invariants in $\QQ$ corresponding to isolated points on $X_1(n)$, as $n$ varies over all positive integers. 

% Moreover, \cite[Theorem 1.1]{Bou19} states that for a non-CM elliptic curve with $m$-adic Galois representation of level $M$, for a certain integer $m$, the natural map $X_1(n) \to X_1((n,M))$ sends isolated points to isolated points. This gives a method of reducing the level on which an isolated point must appear. 
% We employ a similar strategy in Section \ref{sec:prime-power level}, in which, given a modular curve $X_H$ with \notice, we find a modular curve $X_{H'}$ such that the natural map $X_H \to X_{H'}$ sends isolated points to isolated points. 

\par
It is natural to ask for which other families of curves we can obtain similar finiteness results. 
% similar to \cite[Corollary 1.7]{Bou19}. 
In this paper, we examine the set of $j$-invariants in an extension of bounded degree that arise as the $j$-invariant of an isolated point on a modular curve of prime-power level.
\begin{theorem} \label{thm:intro:fixed prime power finiteness}
    Fix a prime $\ell \in \NN$ and define
    \[ \cI \defeq \set{x \in X_H}{x \text{ is isolated and } H \leq \GL_2(\ZZ_\ell) \text{ an open subgroup} }. \]
    For every $d \in \NN$, there are finitely many $j$-invariants of degree $d$ in $j(\cI)$. 
\end{theorem}

\par 
The proof uses group theory techniques and results of Cadoret and Tamagawa \cite{Cad13} and Bourdon, Ejder, Liu, Odumodu, and Viray \cite{Bou19} to obtain a bound $m$, dependent on $\ell$ and $d$, such that for any isolated point on an $\ell$-power level modular curve with $j$-invariant of degree $d$, there exists another isolated point on a modular curve of $\ell$-power level dividing $m$ having the same $j$-invariant. 
This method of reducing the level of modular curve on which an isolated point must appear is related to \cite[Theorem 1.1]{Bou19}, which states that for a non-CM elliptic curve with $m$-adic Galois representation of level $M$, for a certain integer $m$, the natural map $X_1(n) \to X_1(\gcd(n,M))$ sends isolated points to isolated points. 
We also use results of Terao \cite{Ter24} to describe the degrees of closed points $x \in X_H$ in terms of $H$ and the adelic image of Galois associated to an elliptic curve corresponding to $x$ defined over $\QQ(j(x))$.

% \par
% Here is an explicit example to better illustrate the set $\cI$ appearing in Theorem \ref{thm:intro:fixed prime power finiteness}.
\begin{remark}
    Terao \cite[Theorem 1.5]{Ter24} showed that if $x$ is a non-cuspidal, non-CM isolated point on a modular curve $X_H$, with $H \leq \GL_2(\hat{\ZZ})$ an open subgroup of level $7$ and $j(x) \in \QQ$, then $j(x) = 3^3\cdot 5 \cdot 7^5/2^7$. 
    % If we instead ask which $j$-invariants $j(x)$ appear in an extension of fixed bounded degree, with $H$ having level a power of $7$, then Theorem \ref{thm:intro:fixed prime power finiteness} ensures there will be only finitely many. 
    % If we instead ask which rational $j$-invariants $j(x)$ appear in an extension of fixed bounded degree, with $H$ having level a power of $7$, then Theorem \ref{thm:intro:fixed prime power finiteness} ensures there will be only finitely many. 
    Theorem \ref{thm:intro:fixed prime power finiteness} implies that there are finitely many $j$-invariants of fixed degree arising from isolated points on modular curves whose level is any power of $7$.
\end{remark}

\par 
Recall the modular curve $X_0(n)$ is a curve whose non-cuspidal points parametrize elliptic curves with a cyclic subgroup of $n$. 
In the case of prime-power level curves in the families $X_0(n)$ and $X_1(n)$, Bourdon and Ejder \cite{Bou25-2} give an unconditional result classifying the rational $j$-invariants arising from isolated points that may appear, which extends a partial classification given by Ejder \cite{Ejd22}. 
% Bourdon and Ejder show that if $x \in X_1(\ell^n)$ is an isolated point with 
Building on these results, we study the rational $j$-invariants arising from isolated points on prime-power level curves in the
family $X_\Delta(n)$, which contains the families $X_0(n)$ and $X_1(n)$ and the so-called intermediate modular curves living between $X_0(n)$ and $X_1(n)$. 

\begin{theorem} \label{thm:intro:X Delta prime power}
    Let $\ell \in \NN$ be a prime, let $n \in \NN$ a positive number, and suppose $\Delta \leq (\ZZ/\ell^n \ZZ)^\times$ is a subgroup containing $-1$. 
    If $x \in X_\Delta(\ell^n)$ is a non-cuspidal isolated point with non-CM rational $j$-invariant $j \in \QQ$, then $j$ appears as the $j$-invariant of an isolated point on $X_0(\ell)$, for some $\ell \in \{\, 11, 17, 37 \,\}$. 
\end{theorem}

The proof uses the work of Bourdon and Ejder \cite{Bou25-2} and the work of Rouse, Sutherland, and Zureick-Brown \cite{RSZB22} and Furio \cite{Fur25} studying $\ell$-adic images of Galois attached to non-CM elliptic curves over $\QQ$, to determine whether a rational $j$-invariant $j(E)$ associated to a non-CM elliptic curve $E/\QQ$ may appear as the $j$-invariant of an isolated point on a modular curve $X_\Delta(\ell^n)$. 

\begin{remark}
    If we instead consider rational $j$-invariants arising from isolated points on $X_\Delta(n)$, for $n$ not necessarily a prime power, then the analogue of Theorem \ref{thm:intro:X Delta prime power} no longer holds, as noted by Lee \cite[Theorem 2]{Lee25}. Specifically, Bourdon, Hashimoto, Keller, Klagsbrun, Lowry-Duda, Morrison, Najman, and Shukla \cite[Theorem 2]{Bou24} showed that there is an isolated point on $X_1(28)$ with $j$-invariant $351/4$, but Lee found that there is no isolated point on a modular curve $X_0(n)$ with $j$-invariant $351/4$. 
\end{remark}

% The proof of Theorem \ref{thm:intro:X Delta prime power} also shows \comment{what IMCs the $j$-invariants occur on}

% \par
% \comment{\& other (code acknowledgements, outline, etc.)}

\subsection{Related work}
\par
This is also related to work of Menendez \cite[Theorem 5.3]{Men22}, who proved the analogue to \cite[Theorem 1.1]{Bou19} for the family $X_0(n)$. 
Lee \cite[Algorithm 1]{Lee25} gave an algorithm to test whether a non-CM $j$-invariant in $\QQ$ corresponds to an isolated point on a modular curve $X_0(n)$. 
Among all elliptic curves in the LMFDB, 
% If we assume an affirmative answer to Serre's uniformity problem, then 
the $j$-invariants found in \cite[Table 1]{Lee25} are all of the possible $j$-invariants in $\QQ$ corresponding to an isolated point on $X_0(n)$. 
Terao \cite[Theorem 1.6]{Ter24} proved a very similar result, classifying the rational $j$-invariants arising from isolated points on $X_0(n)$, assuming a conjecture of Zywina \cite[Section 14.3]{Zyw24}. The $j$-invariants found by Lee in \cite[Table 1]{Lee25} are the same $j$-invariants appearing in \cite[Theorem 1.6]{Ter24}. 

\subsection{Outline} 
In Section \ref{sec:background}, we give the relevant background on elliptic curves, modular curves, and isolated points. In Section \ref{sec:prime-power level}, we prove Theorem \ref{thm:intro:fixed prime power finiteness}. In Section \ref{sec:IMC}, we give an overview of the family of modular curves $X_\Delta(n)$. In Section \ref{sec:IMC prime-power level}, we prove Theorem \ref{thm:intro:X Delta prime power}. 

\subsection{Acknowledgements}
% \section*{Acknowledgements}
The author was partially supported by NSF grant DMS-2145270. 
This work was carried out under the supervision of Abbey Bourdon, whom the author is thankful to for her guidance. 
The author also thanks Maarten Derickx for helpful conversations.

% \subsection{Outline}
% In Section \ref{sec:background},  \notice

% ...

% \par In Section \ref{sec:prime-power level},  \notice 

% \newpage
\section{Background} \label{sec:background}
% \comment{What should be included in this section?}
% \begin{enumerate}
%     \item notation
%     \item group theory---lemmas and profinite groups
%     \item $H$-level structure
%     \item modular curve
%     \item isolated
% \end{enumerate}

\subsection{Notation and conventions} 

Let $k$ be a number field and $\overline{k}$ a fixed algebraic closure of $k$.
We denote $\Gal_k$ to be the absolute Galois group $\Gal(\overline{k}/k)$.
% For $k$ a number field, we denote $\Gal_k$ to be the absolute Galois group $\Gal(\overline{k}/k)$.
Let $\NN$ denote the set of natural numbers, i.e. the set of positive integers $\{\, 1, 2, 3, \dots \,\}$. For $n \in \NN$, we denote by $\phi(n)$ the number of integers $1 \leq m \leq n$ with $(m,n) = 1$. 

\par 
By a \emph{nice} curve, we mean a smooth, projective, geometrically integral curve. 
For a $k$-scheme $X$ and a field extension $K$ of $k$, we denote $X_{K}$ to be the base change $X \times_{\Spec{k}} \Spec{K}$. 

\subsection{Elliptic curves}

Let $E$ be an elliptic curve over a number field $k$. For $n \in \NN$, we denote $E[n] = E(\overline{k})[n]$ to be the subgroup of $E(\overline{k})$ consisting of $n$-torsion points---that is, points with finite order dividing $n$. The group $E[n]$ is the kernel of the multiplication-by-$n$ map $[n] : E \to E$ and is isomorphic to $(\ZZ/n\ZZ) \times (\ZZ/n\ZZ)$. 

\par 
We say an elliptic curve $E/k$ has \emph{complex multiplication (CM)} if $\End{E_{\overline{k}}}$ is not isomorphic to the integers---we will not add the distinction of an elliptic curve having potential CM, as some authors do. We say a $j$-invariant $j \in \overline{\QQ}$ is CM if there exists a CM elliptic curve $E$ with $j$-invariant equal to $j$. In this paper, we are interested in showing there are finitely many $j$-invariants, with a certain property, in number fields of a fixed bounded degree. Our main techniques to do so work only for non-CM $j$-invariants. However, results from class field theory 
allow us to restrict to the case of non-CM $j$-invariants.

\par 
Indeed, if $E/k$ has CM, then by \cite[Corollary 9.4]{Sil09}, there exists an imaginary quadratic field $K$ 
% with ring of integers $\cO_K$ 
and an order $\cO$ of conductor $f$ in $K$ such that $\End(E_{\overline{k}}) \simeq \cO$. 
By \cite[Theorem 11.1]{Cox13}, we have $[\QQ(j(E)) : \QQ] = [K(j(E)) : K] = h(\cO)$, where $h(\cO)$ is the class number. 
% The class number of $\cO$ is bounded by \cite[Theorem III]{Hei34}
There are only finitely many $K$ of a given class number by \cite{Hei34}
and so, by \cite[Theorem 7.24]{Cox13}, there are only finitely many imaginary quadratic orders of a given class number.
% there are only finitely many orders $\cO_f = \ZZ + f\cO$ with \comment{formula $h(\cO_f) = ? \cdot h_K$ Cox Cor. 7.24?}
Each order is associated to only finitely many CM $j$-invariants by \cite[Corollary 10.20]{Cox13}. These results imply the following well-known theorem.
\begin{theorem} \label{thm:finitely many CM j-invariants in degree d number fields}
    For every $d \in \NN$, there are only finitely many CM $j$-invariants in number fields of degree $d$. 
\end{theorem}

\subsection{Profinite constructions}

\par For a prime $\ell \in \NN$, the ring of \emph{$\ell$-adic integers}, denoted $\ZZ_\ell$, is the inverse limit $\projlim \ZZ/\ell^n \ZZ$ taken over $n \in \NN$ with homomorphisms the evaluation maps $\ZZ/\ell^j \ZZ \to \ZZ/\ell^i \ZZ$ for $i \leq j$. The \emph{profinite integers}, denoted $\hat{\ZZ}$, are the inverse limit $\projlim \ZZ/n\ZZ$ taken over the directed set consisting of natural numbers ordered by divisibility. There is a ring isomorphism
\[ \hat{\ZZ} = \projlim \ZZ/n\ZZ \simeq \prod_{\substack{\ell\in \NN \\ \text{prime}}} \projlim \ZZ/\ell^n \ZZ = \prod_{\ell} \ZZ_\ell . \]

\par For an elliptic curve $E$ over a number field $k$ and a prime $\ell \in \NN$, we define the \emph{$\ell$-adic Tate module of $E$} to be the inverse limit $T_\ell (E) \defeq \projlim E[\ell^n]$ taken over $n \in \NN$ with homomorphisms the maps $[\ell^{j-i}] : E[\ell^j] \to E[\ell^i]$ for $i \leq j$. As a $\ZZ_\ell$-module, $T_\ell(E)$ has the following structure:
\[ T_\ell(E) = \projlim E[\ell^n] \simeq \projlim \left( \ZZ/\ell^n \ZZ \times \ZZ/\ell^n \ZZ \right) = \projlim \left( \ZZ/\ell^n \ZZ \right) \times \projlim \left( \ZZ/\ell^n \ZZ \right) = \ZZ_\ell \times \ZZ_\ell. \]
Following a similar construction to $\hat{\ZZ}$, we also define the \emph{adelic Tate module of $E$} to be the inverse limit $T(E) \defeq \projlim E[n]$ so that 
\[ T(E) \simeq \prod_{\ell} T_\ell(E) \simeq \prod_\ell \ZZ_\ell \times \ZZ_\ell \simeq \hat{\ZZ} \times \hat{\ZZ} .\]

\par We will also study the profinite group $\GL_2(\hat{\ZZ}) \simeq \projlim \GL_2(\ZZ/n\ZZ)$. 
For each $n \in \NN$, there exists a natural projection map $\pi_n : \GL_2(\hat{\ZZ}) \to \GL_2(\ZZ/n\ZZ)$ induced by the inverse limit. Similarly, by abuse of notation, for a prime number $\ell$, we denote $\pi_{\ell^n}$ to be the natural projection map $\GL_2(\ZZ_\ell) \to \GL_2(\ZZ/\ell^n\ZZ)$. 
% For each $n \in \NN$, there exists a projection map $\pi_n : \GL_2(\hat{\ZZ}) \to \GL_2(\ZZ/n\ZZ)$ induced by the natural projection $\prod_{N \in \NN} \GL_2(\ZZ/N\ZZ) \to \GL_2(\ZZ/n\ZZ)$. 
% For $n$ a power of $\ell$, we also denote $\pi_n$ to be the natural map $\GL_2(\ZZ_\ell) \to \GL_2(\ZZ/n\ZZ)$ induced by  
% If $n$ is a power of a prime number $\ell$, then the map $\pi_n$ factors through $\GL_2(\ZZ_\ell)$. 
% In this case, we also denote $\pi_n$ to be the projection $\GL_2(\ZZ_\ell) \to \GL_2(\ZZ/n\ZZ)$. 
For a subgroup $H \leq \GL_2(\hat{\ZZ})$, we denote $H(n)$ to be $\pi_n (H)$. In general, we have $H \leq \pi_n\inv(H(n))$. If $H \leq \GL_2(\hat{\ZZ})$ is an open subgroup, then there exists a positive integer $n \in \NN$ such that $\ker \pi_n \leq H$. We define the \emph{level of $H$} to be the least such positive integer. If $H$ has level $n$, then $H = \pi_n\inv(H(n))$ and the index of $H$ in $\GL_2(\hat{\ZZ})$ (which is finite because $H$ is open) is equal to the index of $H(n)$ in $\GL_2(\ZZ/n\ZZ)$. 

\par
In Section \ref{sec:IMC prime-power level}, we will consider specific subgroups $H \leq \GL_2(\hat{\ZZ})$ with $\det(H) = \hat{\ZZ}^\times$. We will refer to these subgroups using the notation of Rouse, Sutherland, and Zureick-Brown \cite[Section 2.4]{RSZB22}, which assigns to each subgroup $H$ a label
\begin{center}
    \verb|N.i.g.n|
\end{center}
where \verb|N| is the level of $H$, \verb|i| is the index of $H$ in $\GL_2(\hat{\ZZ})$, \verb|g| is the genus of the modular curve $X_H$, and \verb|n| is an integer which specifies $H$ among other subgroups having the same level, index, and genus.

\subsection{Galois representations}

For every $\sigma \in \Gal_k$ and $P \in E$, we denote by $P^\sigma$ the image of the pair $(P, \sigma)$ under the natural (right) action $E \times \Gal_k \to E$. For every $P \in E[n]$, we have $[n] P = O$, where $O$ is the base point of $E$. For every $\sigma\in \Gal_k$, 
\[ [n] \left(P^\sigma \right) = \left( [n] P \right)^\sigma = O^\sigma = O, \]
so the natural action of $\Gal_k$ on $E[n]$ is well-defined. This action defines a representation
\[ \rho_{E,n} : \Gal_k \to \Aut \left( E[n] \right) \]
called the \emph{mod-$n$ Galois representation associated to $E$}.
Let $\alpha : E[n] \to (\ZZ/n\ZZ) \times (\ZZ/n\ZZ)$ be an isomorphism of $(\ZZ/n\ZZ)$-modules. The \emph{mod-$n$ Galois representation associated to $E$ and $\alpha$} is the composition of the above representation with the isomorphism $\Aut(E[n]) \simeq \GL_2(\ZZ/n\ZZ) $ induced by $\alpha$ and is denoted
\[ \rho_{E,n,\alpha} : \Gal_k \to \Aut \left( E[n] \right) \xra{\simeq} \GL_2(\ZZ/n\ZZ) . \]
If $\beta : E[n] \to (\ZZ/n \ZZ) \times (\ZZ/n \ZZ)$ is another isomorphism of groups, then $\rho_{E,n,\alpha}(\Gal_k)$ and $\rho_{E,n,\beta}(\Gal_k)$ are conjugate in $\GL_2(\ZZ/n\ZZ)$. 
% If we do not need to keep track of conjugation in $\GL_2(\ZZ/n\ZZ)$, then we will not record the isomorphism $E[n] \to (\ZZ/n\ZZ) \times (\ZZ/n\ZZ)$ and simply write $\rho_{E,n}$ for the mod-$n$ Galois representation associated to $E$. 

\par
Suppose $\ell \in \NN$ is a prime and $\alpha : T_\ell(E) \to \ZZ_\ell \times \ZZ_\ell$ is an isomorphism. The action of $\Gal_k$ on $E[\ell^n]$ for every $n\in\NN$ induces an action on $T_\ell(E)$. This defines a representation
\[ \rho_{E, \ell^\infty, \alpha} : \Gal_k \to \Aut(T_\ell(E)) \xra{\simeq} \GL_2(\ZZ_\ell) \]
called the \emph{$\ell$-adic Galois representation associated to $E$ and $\alpha$}. 
Similarly, the action of $\Gal_k$ on $E[n]$ for every $n \in \NN$ extends to an action on $T(E)$. If $\alpha : T(E) \to \hat{\ZZ} \times \hat{\ZZ}$ is an isomorphism, then this action induces a representation
\[ \rho_{E, \alpha} : \Gal_k \to \Aut(T(E)) \xra{\simeq} \GL_2(\hat{\ZZ}) \]
called the \emph{adelic Galois representation associated to $E$ and $\alpha$}. 
Just as with the mod-$n$ Galois representation, we may not keep track of the isomorphisms $\alpha$ if we only care about the image of Galois up to conjugation. 

\par
Serre showed \cite[Th\'eor\`eme 3]{Ser72} that for a non-CM elliptic curve $E/k$, the adelic representation of Galois, $\rho_E : \Gal_k \to \GL_2(\hat{\ZZ})$ has open image in $\GL_2(\hat{\ZZ})$. Note that since the coset of an open subgroup of a profinite group is open, $\rho_{E,\alpha}(\Gal_k)$ is open regardless of the isomorphism $\alpha$, as this only changes the image of Galois up to conjugation. 

% \begin{theorem}[{\cite{Ser72}}] \label{thm:Serre open image}
%     Let $E/k$ be a non-CM elliptic curve. Then $\rho_{E,\ell^\infty}(\Gal_k)$ is open in $\GL_2(\ZZ_\ell)$. 
% \end{theorem}

\subsection{Modular curves}

% Let $H$ be an open subgroup of $\GL_2(\ZZ/n\ZZ)$. In essence, the modular curve $X_H$ is a curve whose non-cuspidal points parametrize elliptic curves having mod-$n$ image of Galois contained in $H$. 

In essence, for $H$ an open subgroup of $\GL_2(\ZZ/n\ZZ)$ (resp., $\GL_2(\ZZ_\ell)$; $\GL_2(\hat{\ZZ})$), the modular curve $X_H$ is a curve whose non-cuspidal points parametrize elliptic curves having mod-$n$ (resp., $\ell$-adic; adelic) image of Galois contained in $H$. We now give a more precise definition. 

\par
Let $H$ be a subgroup of $\GL_2(\ZZ/n\ZZ)$ and let $E$ be an elliptic curve defined over a number field $k$. We say two isomorphisms $\alpha, \alpha' : E[n] \to (\ZZ/n\ZZ) \times (\ZZ/n\ZZ)$ are \emph{$H$-equivalent}, denoted $\alpha \sim_H \alpha'$ if there exists $h \in H$ such that $\alpha = h \circ \alpha'$. Let $E$ be an elliptic curve over a number field $k$. An \emph{$H$-level structure on $E$} is an equivalence class $[\alpha]_H$ of isomorphisms $\alpha : E[n] \to (\ZZ/n\ZZ) \times (\ZZ/n\ZZ)$ under $H$-equivalence.

\par
We define the modular curve $Y_H$ (resp. $X_H$) to be the coarse moduli space of the stack $\cM_H^0$ (resp. $\cM_H$), which parametrizes elliptic curves (resp. generalized elliptic curves) with $H$-level structure. The modular curve $Y_H$ is an affine subscheme of $X_H$. We call the s of $X_H - Y_H$ \emph{cusps} and the closed points of $Y_H$ \emph{non-cuspidal points}. The curve $X_H$ is a smooth, projective, integral curve over $\QQ$ and is geometrically integral if and only if $H$ has full determinant. 
% \comment{citation? Bjorn Poonen.}

\par
If $H$ is an open subgroup of $\GL_2(\hat{\ZZ})$ or $\GL_2(\ZZ_\ell)$ of level $n$, then we define the modular curve $X_H$ to be the modular curve $X_{H(n)}$. If $m$ is any positive integer divisible by $n$, then the modular curve $X_H = X_{H(n)}$ is isomorphic to the modular curve $X_{H(m)}$. Even if $H$ does not contain $-I$, we always have that the modular curves $X_H$ and $X_{\pm H}$ are isomorphic as curves. Because of this, we will only consider subgroups $H$ containing $-I$. 

\par
We give a precise description of the geometric non-cuspidal points of $X_H$. 
The set $Y_H(\overline{k})$ consists of equivalence classes of pairs $(E, [\alpha]_H)$, where $E/k$ is an elliptic curve and $[\alpha]_H$ is an $H$-level structure on $E$. We say two pairs $(E, [\alpha]_H)$ and $(E', [\alpha']_H)$ are equivalent if there exists an isomorphism $\phi : E \to E'$ such that the induced isomorphism $\phi : E[n] \to E'[n]$ satisfies $\alpha \sim_H \alpha' \circ \phi $, i.e. $\alpha = h \circ \alpha' \circ \phi$, for some $h \in H$. 

\par
The absolute Galois group $\Gal_k$ admits a right action on $Y_H(\overline{k})$ as follows:
For an automorphism $\sigma \in \Gal_k$ and a point $[(E, [\alpha]_H)] \in Y_H(\overline{k})$, we define 
\[ [(E,[\alpha]_H)] \cdot \sigma = [(E^\sigma, [\alpha \circ \sigma\inv]_H)]. \]
This action is well-defined: Indeed, if $(E', [\alpha']_H), [(E, [\alpha]_H)] \in Y_H(\overline{k})$ and $\sigma \in \Gal_k$, then there exists an isomorphism $\phi : E[n] \to E'[n]$ and an element $h \in H$ such that $\alpha = h \circ \alpha' \circ \phi$. This implies 
\[ \alpha \circ \sigma\inv = (h \circ \alpha' \circ \phi) \circ \sigma \inv = h \circ (\alpha' \circ \sigma \inv) \circ (\sigma \circ \phi \circ \sigma \inv) . \]
The map $\sigma \circ \phi \circ \sigma \inv : E^\sigma[n] \to E'^\sigma [n]$ given by $P \mapsto \phi(P^{\sigma\inv})^\sigma$ is an isomorphism, so 
\[ ((E', [\alpha']_H) \cdot \sigma) = (E'^\sigma, [\alpha' \circ \sigma \inv]_H)  \in [(E^\sigma, [\alpha \circ \sigma\inv]_H)] = [(E, [\alpha]_H) \cdot \sigma] . \] 

\par 
A \emph{closed point} of $X_H$ is a point $x \in X_H$ such that $\{\, x \,\}$ is Zariski closed in $X$ and the \emph{degree} (over $\QQ$) of the closed point, denoted $\deg(x)$, is the degree of the residue field of $x$ over $\QQ$. By {\cite[Proposition 2.4.6]{Poo17}}, there is a one-to-one correspondence between closed points of $X_H$ and $\Gal_k$-orbits of points in $X_H(\overline{k})$ and the degree of a closed point corresponds to the size of the $\Gal_k$-orbit. We will often refer to such a $\Gal_k$-orbit as a closed point. 

\par 
We say a point in $X_H(\overline{k})$ is \emph{$k$-rational} if it is fixed by every element of $\Gal_k$. Equivalently, a $k$-rational point is a closed point of degree $1$.
% We say a point in $X_H(\overline{k})$ is \emph{$k$-rational} if the closed point corresponding to its $\Gal_k$-orbit has degree $1$. 

\par
We say a pair $(E, [\alpha]_H)$, consisting of an elliptic curve $E/\QQ(j(E))$ and an $H$-level structure $[\alpha]_H$ on $E$, is a \emph{minimal representative} for a closed point $x \in X_H$ if $x$ corresponds to the $\Gal_\QQ$-orbit of  $[(E, [\alpha]_H)] \in X_H(\overline{\QQ})$. A minimal representative for a non-cuspidal closed point always exists by \cite[Lemma 4.9]{Ter24}.

% \par
% If $H_1$ and $H_2$ are subgroups of $\GL_2(\ZZ/n\ZZ)$ containing $-I$ with $H_1 \leq H_2$, then there is a natural inclusion map $X_{H_1} \to X_{H_2}$. Indeed, if $(E, [\alpha]_{H_1})$ is a representative for a non-cuspidal point in $Y_{H_1}(\overline{k})$, then $[\alpha]_{H_1}$ defines an $H_2$-level structure on $E$ because $\alpha \sim_{H_1} \alpha'$ implies $\alpha \sim_{H_2} \alpha'$. We summarize a result from \cite{Dia06}, which describes the degree of this inclusion map. 

% \begin{theorem}[{\cite[Page 66]{Dia06}}] \label{thm:Dia06 p 66:contain -I}
%     Let $H_1, H_2 \leq \GL_2(\hat{\ZZ})$ be open subgroups containing $-1$. Suppose $H_1 \leq H_2$ and denote $f$ to be the natural map $X_{H_1} \to X_{H_2}$ of curves over $\QQ$. Then $\deg(f) = [H_2 : H_1]$. 
% \end{theorem}

\begin{theorem}[{\cite[Theorem 4.24]{Ter24}}] \label{thm:Ter24 4.24}
    Let $H \leq \GL_2 (\hat{\ZZ})$ be an open subgroup and let $x \in X_H$ be a non-cuspidal closed point with minimal representative $(E, [\alpha]_H)$ defined over $k \defeq \QQ(j(E))$. Let $A_{E,\alpha} \leq \GL_2(\hat{\ZZ})$ be the subgroup $\set{\alpha \circ \phi \circ \alpha\inv}{\phi \in \Aut(E_{\overline{k}})}$.
    Then 
    \[ \deg(x) = [\QQ(j(E)) : \QQ][\rho_{E,\alpha}(\Gal_k) A_{E,\alpha} : \rho_{E,\alpha}(\Gal_k) A_{E,\alpha} \cap A_{E, \alpha} H]. \]
\end{theorem}
If $j(x) \neq 0, 1728$---in particular, if $E$ is non-CM---then $\Aut(E_{\overline{k}}) = \{\, \pm 1 \,\}$, so $A_{E,\alpha} = \{\, \pm I \,\}$. Since we are assuming $H$ contains $-I$, we get that $A_{E,\alpha}H = \{\, \pm I \,\} H = H$. This yields:
\begin{corollary} \label{cor:deg(x) on X_H}
     Let $H \leq \GL_2 (\hat{\ZZ})$ be an open subgroup containing $-I$ and let $x \in X_H$ be a non-cuspidal closed point with minimal representative $(E, [\alpha]_H)$ defined over $k \defeq \QQ(j(E))$ such that $j(E) \neq 0, 1728$. Then 
     \[ \deg(x) = [\QQ(j(E)) : \QQ][\rho_{E, \alpha}(\Gal_k) : \rho_{E, \alpha}(\Gal_k) \cap H] . \]
\end{corollary}

\par
If $H_1$ and $H_2$ are subgroups of $\GL_2(\ZZ/n\ZZ)$ containing $-I$ with $H_1 \leq H_2$, then there is a natural inclusion map $X_{H_1} \to X_{H_2}$. Indeed, if $(E, [\alpha]_{H_1})$ is a representative for a non-cuspidal point in $Y_{H_1}(\overline{k})$, then $[\alpha]_{H_1}$ defines an $H_2$-level structure on $E$ because $\alpha \sim_{H_1} \alpha'$ implies $\alpha \sim_{H_2} \alpha'$. We summarize a result from \cite{Dia06}, which describes the degree of this inclusion map. 

\begin{theorem}[{\cite[Page 66]{Dia06}}] \label{thm:Dia06 p 66:contain -I}
    Let $H_1, H_2 \leq \GL_2(\hat{\ZZ})$ be open subgroups containing $-1$. Suppose $H_1 \leq H_2$ and denote $f$ to be the natural map $X_{H_1} \to X_{H_2}$ of curves over $\QQ$. Then $\deg(f) = [H_2 : H_1]$. 
\end{theorem}

\subsection{Isolated points} \label{subsection:isolated points}

\par
Let $C$ be a smooth, projective curve over a number field $k$. 
% Let $\Pic(C)$ be the Picard group 
% Let $\Pic(C)$ be the Picard group of $C$, let $\PPic_{C/k}$ be the Picard scheme of $C$, and let $\PPic_{C/k}^0$ be the connected component of the identity. 
Let $\PPic_{C/k}$ be the Picard scheme of $C$ and let $\PPic_{C/k}^0$ be the connected component of the identity. 
% Denote by $\Div(C/k)$ the set of relative effective Cartier divisors on $C/k$, and by $\DDiv_{C/k}$ the divisor scheme of $C$. 
% \comment{first part needed?}
Denote $\DDiv_{C/k}$ to be the divisor scheme of $C$. 
Denote by $\AA_{C/k} : \DDiv_{C/k} \to \PPic_{C/k}$ the Abel map, and by $\WW_{C/k}$ its image. 
We refer the reader to \cite[Section 2]{Ter24} for more details. 

% , which by \cite[Remark 2.1]{Ter24} 
% The divisor scheme of $C$ is a $k$-scheme $\DDiv_{X/k}$ such that $\DDiv_{X/k}(T) = \Div(X_T/T)$ for all $k$-schemes $T$. 

\par 
Let $x \in C$ be a closed point. By taking $x$ to be a sum of $\Gal_k$-conjugates, we can view $x$ as an element of $\DDiv_{C/k}(k)$. 
% with corresponding \notice relative effective Cartier divisor 
% % $D \in \DDiv_{C/k}(k)$. 
% $D \in \Div(C/k)$. 
\begin{enumerate}
    \item We say $x$ is \emph{$\PP^1$-parametrized} if 
    there exists $x' \in \DDiv_{C/k}(k)$ with $x' \neq x$ such that $\AA_{C/k}(x) = \AA_{C/k}(x')$.
    We say $x$ is \emph{$\PP^1$-isolated} if it is not $\PP^1$-parametrized.
    \item We say $x$ is \emph{AV-parametrized} if 
    there exists a positive rank abelian subvariety \break ${A \subseteq \PPic_C^0}$ such that $\AA_{C/k}(x) + A \subseteq \WW_{C/k}$.
    We say $x$ is \emph{AV-isolated} if it is not AV-parametrized.
    \item We say $x$ is \emph{isolated} if it is both $\PP^1$-isolated and AV-isolated.
    \item We say $x$ is \emph{sporadic} if there are finitely many points $y \in C$ with $\deg(y) \leq \deg(x)$. 
\end{enumerate}

\par 
The study of isolated points is motivated by Faltings's theorem \cite{Fal83} which says that if $C/k$ has genus strictly greater than $1$, then $C(k)$ is finite and hence every point in $C(k)$ is sporadic. Indeed, it was shown in \cite[Theorem 4.2]{Bou19} that on $C/k$, every sporadic point is isolated and there are only finitely many isolated points. So, if $C/k$ has genus strictly greater than $1$, the study of isolated points is a natural extension of the study of rational points. 

\par
The following theorem generalizes \cite[Lemma 2.3]{Ejd22} and will be particularly useful to us in Section \ref{sec:IMC prime-power level}. 

\begin{theorem}[{\cite[Theorem 2.17]{Ter24}}] \label{thm:RR isolated}
    Let $C$ be a smooth, projective curve over a number field $k$ and let $K \defeq k(C) \cap \overline{k}$. Let $r = [K : k]$ be the number of geometric components of $C$ and let $g$ be the genus of $C_K$. If $x \in C$ is a closed point with $\deg(x) > rg$, then $x$ is $\PP^1$-parametrized. 
\end{theorem}

\par 
% We now consider how isolated points behave 
For a finite locally free map of curves, we have the following inequality:

\begin{lemma} \label{lemma:closed point deg leq}
    Let $C$ and $D$ be smooth, projective curves over a number field $k$ and let $f : C \to D$ be a finite locally free map. If $x \in C$ is a closed point, then 
    \[ \deg(x) \leq \deg(f) \deg(f(x)) . \]
\end{lemma}
% \notice 
\begin{proof}
    Let $x \in C$ be a closed point. Let $f^* : \DDiv_{D/k} \to \DDiv_{C/k}$ be the pullback map induced by $f$ and let $f_* : \DDiv_{C/k} \to \DDiv_{D/k}$ be the pushforward map induced by $f$. By \cite[Proposition 21.10.4]{Gro67}, we have 
    \[\deg(f^* (f(x))) = \sum_{x' \in f\inv(f(x))} e_f(x')\deg(x'),\] 
    where $e_f(x')$ is the ramification index of $f$ at $x'$. By \cite[Proposition 21.10.18]{Gro67}, we have $f_*(f^*(f(x))) = \deg(f) f(x)$, so 
    \[ \deg(f^*(f(x))) = \deg(f_*(f^*(f(x)))) = \deg(f) \deg(f(x)) . \] 
    This implies
    \[ \deg(x) \leq \sum_{x' \in f\inv(f(x))} e_f(x')\deg(x') = \deg(f) \deg(f(x)) . \qedhere \]
    % \[ \deg(f) \deg(f(x)) = \deg(f^*(f(x))) = \sum_{x' \in f\inv(f(x))} e_f(x')\deg(x') \geq \deg(x) . \]
    % https://webusers.imj-prg.fr/~leila.schneps/grothendieckcircle/EGA/EGAIV.4.pdf
\end{proof}
% {\cite[Theorem 19.22]{Sut13}}, 
% \[ \deg(f) \deg(f(x)) = \deg(f^* (x)) = \sum_{y \in f\inv(f(x))} e_f(y)\deg(y) \geq \deg(x) . \]

\par If $f : C \to D$ is a finite map of curves $x \in C$ is an isolated point, it is not true in general that the image $f(x)$ is an isolated point of $D$. However, if the degree is as large as possible---that is, if $\deg(x) = \deg(f)\deg(f(x))$---then we can in fact conclude that $f(x)$ is isolated. This result was first proved for nice curves in {\cite[Theorem 4.3]{Bou19}} and then generalized by Terao to the setting of smooth, projective curves. 

\begin{theorem}[{\cite[Theorem 2.15]{Ter24}}] \label{thm:Bou19 4.3}
    Let $f : C \to D$ be a finite locally free map of smooth, projective curves over a number field $k$. If $x \in C$ and $y \in D$ are closed points such that $y = f(x)$ and $\deg(x) = \deg(y) \deg(f)$, then the following hold:
    \begin{enumerate}
        \item If $x$ is $\PP^1$-isolated, then $y$ is $\PP^1$-isolated.
        \item If $x$ is AV-isolated, then $y$ is AV-isolated.
    \end{enumerate}
    In particular, if $x$ is isolated, then $y$ is also isolated.
\end{theorem}

\section{Modular curves of prime-power level} \label{sec:prime-power level}

In this section we prove Theorem \ref{thm:intro:fixed prime power finiteness}. 
Our strategy is to bound the level of modular curves on which a non-CM isolated point with $j$-invariant of fixed degree must appear. 
% This is proved in the following proposition. 

\begin{proposition} \label{prop:fixed prime power deg}
    Fix positive integers $\ell,m \in \NN$ with $\ell$ prime. Let $H \leq \GL_2(\ZZ_\ell)$ be an open subgroup containing $-I$ of level $\ell^n$ with $n \geq m$. Suppose $x \in X_H$ is a non-cuspidal closed point with minimal representative $(E,[\alpha]_H)$ such that $E$ is non-CM and $\rho_{E,\alpha,\ell^\infty}(\Gal_{\QQ(j(E))}) \leq \GL_2(\ZZ_\ell)$ has level dividing $\ell^m$. Let $H' \defeq \pi_{\ell^m}\inv (H(\ell^m))$ and let $f$ denote the natural map $X_H \to X_{H'}$. Then $\deg(x) = \deg(f)\deg(f(x))$.
\end{proposition}

\subsection{Preliminary results}
We shall need the following lemma.

\begin{lemma} \label{lemma:H' = HK and surjective image index}
    Let $H$ be a subgroup of a group $G$ and let $f : G \to G'$ be a surjective homomorphism of groups.
    \begin{enumerate}
        \item If $H$ has finite index in $G$, then $[G : H] \geq [f(G) : f(H)]$. \label{lemma:surjective image index}
        \item If $H' \defeq f\inv(f(H))$ and $K \defeq \ker(f)$, then $H' = HK$. \label{lemma:H' = HK}
    \end{enumerate}
\end{lemma}

\begin{proof}\ 
    \begin{enumerate}
        \item Let $G/H$ denote the set of all cosets of $H$ in $G$. Define a function $\phi : G/H \to f(G)/f(H)$ by $xH \mapto f(x)f(H)$. First, we show that $\phi$ is well-defined. Suppose $xH = yH$. Then $y\inv x \in H$, so $f(y\inv x) \in f(H)$. Then 
        \[ f(H) = f(y\inv x)f(H) = f(y)\inv f(x) f(H) \,, \]
        which implies 
        \[ \phi(xH) = f(x)f(H) = f(y)f(H) = \phi(yH) . \]
        Thus, $\phi$ is well-defined. Moreover, $\phi$ is surjective because $f$ is surjective. Hence, we conclude
        \[ [G : H] = \#(G/H) \geq \#(f(G)/f(H)) = [f(G) : f(H)] . \]
        \item First, we show $HK \subseteq H'$. Let $x \in HK$ be given and suppose $x = hk$ for some $h \in H$ and $k \in K$. Then 
        \[f(x) = f(hk) = f(h)f(k) = f(h) \in f(H). \]
        Thus, $x \in H'$.
        
        \par Now we show $H' \subseteq HK$. Let $x \in H'$ be given. Then $f(x) \in f(H)$, so there exists $h \in H$ with $f(x) = f(h)$. Notice 
        \[ 1 = f(x) f(h)\inv = f(x h\inv)\,, \]
        so $xh\inv \in K$. Thus, $x \in KH = HK$. \qedhere
    \end{enumerate}
\end{proof}

\subsection{Proof of Proposition \ref{prop:fixed prime power deg}}

By Lemma \ref{lemma:closed point deg leq}, we know $\deg(x) \leq \deg(f) \deg(f(x))$, so it suffices to show 
\begin{equation}
    \frac{\deg(x)}{\deg(f(x))} \geq \deg(f) . \label{eqn:fixed prime power prop}
\end{equation}
Let $\pi_{\ell^\infty}$ denote the projection 
\[ \GL_2(\hat{\ZZ}) \xra{\simeq} \prod_{p} \GL_2(\ZZ_p) \to \GL_2(\ZZ_\ell) .\]
Since $H$ contains $-I$, Theorem \ref{thm:Dia06 p 66:contain -I} implies 
$\deg(f) = [\pi_{\ell^\infty}\inv(H') : \pi_{\ell^\infty}\inv(H)] = [H' : H]$.
Then
\begin{align*}
    \frac{\deg(x)}{\deg(f(x))}
    &= \frac{[\QQ(j(E)) : \QQ][\im \rho_{E,\alpha} : \im \rho_{E,\alpha} \cap \pi_{\ell^\infty}\inv (H)]}{[\QQ(j(E)) : \QQ][\im \rho_{E,\alpha} : \im \rho_{E,\alpha} \cap \pi_{\ell^\infty}\inv(H')]} & \text{(By Corollary \ref{cor:deg(x) on X_H})}
    \\ &= [\im \rho_{E,\alpha} \cap \pi_{\ell^\infty}\inv(H') : \im \rho_{E,\alpha} \cap \pi_{\ell^\infty}\inv(H)] 
    \\ &\geq [\pi_{\ell^\infty}(\im \rho_{E,\alpha} \cap \pi_{\ell^\infty}\inv(H') ): \pi_{\ell^\infty}(\im \rho_{E,\alpha} \cap \pi_{\ell^\infty}\inv(H))] & \text{(By Lemma 
    % \hyperref[lemma:surjective image index]{\ref*{lemma:H' = HK and surjective image index}.\ref*{lemma:surjective image index}}
    \enumref{lemma:surjective image index}{lemma:H' = HK and surjective image index})}
    \\ &= [\im \rho_{E,\alpha, \ell^\infty} \cap H' : \im \rho_{E,\alpha, \ell^\infty} \cap H] . 
\end{align*}

% \par Denote $R \defeq \rho_{E,\alpha,\ell^\infty}(\Gal_{\QQ(j(E))})$ and note $[R : \GL_2(\ZZ_\ell)]$ is finite because $E$ is non-CM. Consider the following subset lattice:
\par Denote $R \defeq \rho_{E,\alpha,\ell^\infty}(\Gal_{\QQ(j(E))})$ and consider the following subset lattice: 
% \comment{figure may be unnecessary}
% https://q.uiver.app/#q=WzAsOCxbMSw1LCJcXG9wZXJhdG9ybmFtZXtHTH1fMihcXG1hdGhiZntafV9cXGVsbCkiXSxbMCwxLCJIIl0sWzAsMiwiSCciXSxbMiwyLCJSIl0sWzEsMCwiUiBcXGNhcCBIIl0sWzEsMSwiUiBcXGNhcCBIJyJdLFsxLDQsIkgnUiJdLFsxLDMsIkhSIl0sWzEsMiwiIiwwLHsic3R5bGUiOnsiaGVhZCI6eyJuYW1lIjoibm9uZSJ9fX1dLFsxLDQsIiIsMCx7InN0eWxlIjp7ImhlYWQiOnsibmFtZSI6Im5vbmUifX19XSxbNSw0LCIiLDAseyJzdHlsZSI6eyJoZWFkIjp7Im5hbWUiOiJub25lIn19fV0sWzIsNSwiIiwwLHsic3R5bGUiOnsiaGVhZCI6eyJuYW1lIjoibm9uZSJ9fX1dLFszLDUsIiIsMCx7InN0eWxlIjp7ImhlYWQiOnsibmFtZSI6Im5vbmUifX19XSxbMyw0LCIiLDAseyJzdHlsZSI6eyJoZWFkIjp7Im5hbWUiOiJub25lIn19fV0sWzMsNywiIiwxLHsic3R5bGUiOnsiaGVhZCI6eyJuYW1lIjoibm9uZSJ9fX1dLFs2LDMsIiIsMSx7InN0eWxlIjp7ImhlYWQiOnsibmFtZSI6Im5vbmUifX19XSxbNiw3LCIiLDEseyJzdHlsZSI6eyJoZWFkIjp7Im5hbWUiOiJub25lIn19fV0sWzYsMCwiIiwxLHsic3R5bGUiOnsiaGVhZCI6eyJuYW1lIjoibm9uZSJ9fX1dLFs3LDEsIiIsMSx7InN0eWxlIjp7ImhlYWQiOnsibmFtZSI6Im5vbmUifX19XSxbNiwyLCIiLDEseyJzdHlsZSI6eyJoZWFkIjp7Im5hbWUiOiJub25lIn19fV1d
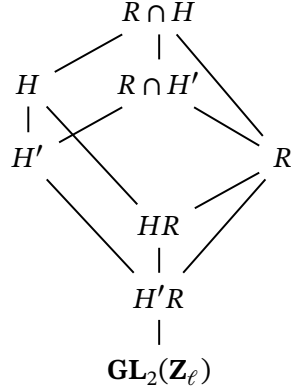
\[\begin{tikzcd}[ampersand replacement=\&,sep=small]
    \& {R \cap H} \\
    H \& {R \cap H'} \\
    {H'} \&\& R \\
    \& HR \\
    \& {H'R} \\
    \& {\GL_2(\mathbf{Z}_\ell)}
    \arrow[no head, from=2-1, to=1-2]
    \arrow[no head, from=2-1, to=3-1]
    \arrow[no head, from=2-2, to=1-2]
    \arrow[no head, from=3-1, to=2-2]
    \arrow[no head, from=3-3, to=1-2]
    \arrow[no head, from=3-3, to=2-2]
    \arrow[no head, from=3-3, to=4-2]
    \arrow[no head, from=4-2, to=2-1]
    \arrow[no head, from=5-2, to=3-1]
    \arrow[no head, from=5-2, to=3-3]
    \arrow[no head, from=5-2, to=4-2]
    \arrow[no head, from=5-2, to=6-2]
\end{tikzcd}\]
The set $HR$ need not be a group. However, $HR$ is a union of left cosets of $R$, the number of which we denote by $[HR : R]$. 
% The index of $R \cap H$ in $\GL_2(\ZZ_\ell)$ is at most $\ell^{m+n}$ and the index of $R$ in $\GL_2(\ZZ_\ell)$ is finite because $E$ is non-CM, so all sets in the above lattice have finite index. 
By assumption, $R$ and $H$ both have finite index in $\GL_2(\ZZ_\ell)$, so all subgroups in the above lattice have finite index. 
The map $H/(R \cap H) \to HR/R$ defined by ${h(R \cap H) \mapto hR}$ is a well-defined bijection of sets, so $[H : R \cap H] = [HR : R]$. Similarly, we have ${[H' : R \cap H'] = [H'R : R]}$. 

\par
Notice that Lemma 
\enumref{lemma:H' = HK and surjective image index}{lemma:H' = HK}
implies $H' = H (\ker \pi_{\ell^m})$. 
Moreover, $R$ contains $\ker \pi_{\ell^m}$ because $R$ has level dividing $\ell^m$. This implies $H'R = H(\ker \pi_{\ell^m}) R = HR$. 
Then,
\begin{align*}
    [R \cap H' : R \cap H]
    &= [H' : H] \frac{[H : R \cap H]}{[H' : R \cap H']}
    \\ &= [H' : H] \frac{[HR : R]}{[H'R : R]}
    \\ &= [H' : H] \frac{[HR : R]}{[HR : R]}
    \\ &= [H' : H] .
\end{align*}
Thus, we have verified inequality (\ref{eqn:fixed prime power prop}).

\subsection{Proof of Theorem \ref{thm:intro:fixed prime power finiteness}}

% We will restate Theorem \ref{thm:intro:fixed prime power finiteness} for convenience.
% To make use of Proposition \ref{prop:fixed prime power deg}, we require a uniform bound $\ell^m$ on the 

The proof of Theorem \ref{thm:intro:fixed prime power finiteness} relies on a result of Cadoret and Tamagawa \cite[Theorem 1.1]{Cad13}, which, for a fixed prime $\ell$ and a fixed positive integer $d$, gives a uniform bound on the level of $\im \rho_{E, \ell^\infty}$, 
% on the index of $\im \rho_{E, \ell^\infty}$ in $\GL_2(\ZZ_\ell)$, 
% \[ [\GL_2(\ZZ_\ell) : \im \rho_{E, \ell^\infty}], \]
as $E$ varies over all non-CM elliptic curves defined over number fields of degree $d$.

\begin{theorem} \label{thm:fixed prime power finiteness}
    Fix a prime $\ell \in \NN$ and define
    \[ \cI \defeq \set{x \in X_H}{x \text{ is isolated and } H \leq \GL_2(\ZZ_\ell) \text{ an open subgroup}}. \]
    For every $d \in \NN$, there are finitely many $j$-invariants of degree $d$ in $j(\cI)$.
\end{theorem}

\begin{proof}
    Let $x \in \cI$ be given. Then $x \in X_H$ is isolated for some subgroup $H \leq \GL_2(\ZZ_\ell)$ of level $\ell^n$. Since $X_H$ and $X_{\pm H}$ are isomorphic as curves, we can assume $H$ contains $-I$. Define $k \defeq \QQ(j(x))$ and $d \defeq [k : \QQ]$.  Let $(E, [\alpha]_H)$ be a minimal representative for $x$. By Theorem \ref{thm:finitely many CM j-invariants in degree d number fields}, there are only finitely many CM $j$-invariants in number fields of degree $d$. As such, it suffices to consider the case where $E$ is non-CM. 
    
    \par By \cite[Theorem 1.1]{Cad13}, there exists a number $m\in \NN$ such that for every degree $d$ number field $k'$ and every non-CM elliptic curve $E'/k'$, the image of the $\ell$-adic Galois representation associated to $E'/k'$, 
    \[ \rho_{E',\ell^\infty} : \Gal_{k'} \to \Aut(T_\ell(E)) \simeq \GL_2(\ZZ_\ell) \]
    has level dividing $\ell^m$. 
    
    \par For each $n' \leq m$, there exist finitely many subgroups of $\GL_2(\ZZ/\ell^{n'}\ZZ)$ and thus finitely many modular curves of level $\ell^{n'}$, each of which has finitely many isolated points by {\cite[Theorem 4.2]{Bou19}}. So, we may assume $n > m$. 
    
    \par Let $H' \leq \GL_2(\ZZ_\ell)$ be the subgroup $\pi_{\ell^m}\inv (H(\ell^m))$ and let $f : X_H \to X_{H'}$ denote the natural inclusion map. By Proposition \ref{prop:fixed prime power deg},
    \[ \deg(x) = \deg(f) \deg(f(x)) . \]
    Then Theorem \ref{thm:Bou19 4.3} implies  $f(x)$ is isolated. 
    
    \par Thus, $j(x)$ is the $j$-invariant of an isolated point on a modular curve of level dividing $\ell^m$. There are only finitely many such modular curves and each has finitely many isolated points. Since $m$ was dependent only on $d$, we conclude $j(\cI)$ contains finitely many $j$-invariants of degree $d$. 
\end{proof}

% \comment{discussion about serre's uniformity problem}

% It is natural to ask if Theorem \ref{thm:fixed prime power finiteness} holds as $\ell$ varies over all primes. That is, if for every $d \in \NN$, there are finitely many $j$-invariants of degree $d$ in 
% \[ \set{j(x)}{x \in X_H \text{ is isolated for some open subgroup } H \leq \GL_2(\ZZ_\ell)}. \]
% This is in fact a stronger claim than an affirmative answer to Serre's uniformity problem. 
% If Serre's uniformity problem is not true, then there exist \notice 

It is natural to ask if Theorem \ref{thm:fixed prime power finiteness} holds as $\ell$ varies over all primes. 
% Consider the case of rational $j$-invariants. That is, if there are finitely many $j$-invariants in the set
% \[ \QQ \cap \set{j(x)}{x \in X_H \text{ is isolated for some open subgroup } H \leq \GL_2(\ZZ_\ell)}. \]
This is in fact a stronger claim than an affirmative answer to Serre's uniformity problem. 
Indeed, suppose the set $\cJ$ of rational isolated $j$-invariants on modular curves of prime-power level is finite. In particular, this implies there are only finitely many rational $j$-invariants corresponding to isolated points on $X_{\Cns(\ell)}$, as $\ell$ ranges over all prime numbers. For each non-CM $j \in \cJ$, there exists a non-CM elliptic curve $E/\QQ$ with $j(E) = j$. Serre's open image theorem \cite[Th\'eor\`eme 3]{Ser72} ensures that for sufficiently large primes $\ell$, the mod-$\ell$ representation of Galois associated to $E$ is surjective. So, there are only finitely-many primes $\ell$ with $\rho_{E,\ell}(\Gal_\QQ)$ conjugate to a subgroup of $\Cns(\ell)$. This implies that $j$ arises as an isolated point on only finitely many modular curves $X_{\Cns(\ell)}$, each of which has finitely many isolated points. 
% \comment{connect to SUC}

\section{Intermediate modular curves} \label{sec:IMC}
% \section{The family of modular curves \texorpdfstring{$X_\Delta(n)$}{X\_Delta(n)}} \label{sec:IMC}
Let $n$ be a positive integer and let $\Delta$ be a subgroup of $(\ZZ/n\ZZ)^\times$.
We define a subgroup of $\GL_2(\ZZ/n\ZZ)$ as follows:
\[ B_\Delta(n) \defeq \set{\twomat{\delta}{a}{0}{b}}{\delta \in \Delta ,\ a \in \ZZ/n\ZZ,\  b \in (\ZZ/n\ZZ)^\times}.  \]
Moreover, we denote $B_1(n)$ (resp. $B_{\pm 1}(n)$; $B_0(n)$) to be $B_\Delta(n)$, with $\Delta = \{\,1\,\}$ (resp. $\Delta = \{\, \pm 1 \,\}$; $\Delta = (\ZZ/n\ZZ)^\times$). 

We will make frequent use of the following formulae. 

\begin{proposition} \label{prop:borel formulae}
    Let $n \in \NN$ have prime factorization $p_1^{a_1} \cdots p_r^{a_r}$ and let $\Delta$ be a subgroup of $(\ZZ/n\ZZ)^\times$. 
    \begin{enumerate}
        \item $\#\GL_2(\ZZ/n\ZZ) = \prod_{i=1}^r p_i^{4a_i-3} \p{p_i^2 - 1}\p{p_i - 1}$. \label{prop:borel formulae:gl2}
        \item $\# B_\Delta(n) = (\#\Delta) \prod_{i=1}^r p_i^{2a_i - 1} \p{p_i - 1}$. \label{prop:borel formulae:B Delta}
        \item $[\GL_2(\ZZ/n\ZZ) : B_\Delta(n)] = \frac{1}{\#\Delta} \prod_{i=1}^r p_i^{2a_i-2} \p{p_i^2 - 1}$. \label{prop:borel formulae:index B Delta}
    \end{enumerate}
\end{proposition}

\begin{proof}
    Formula (\ref{prop:borel formulae:gl2}) follows from the short exact sequence
    \[ 1 \to \SL_2(\ZZ/n\ZZ) \xmra{} \GL_2(\ZZ/n\ZZ) \xera{\det} (\ZZ/n\ZZ)^\times \to 1 ,\]
    formula (\ref{prop:borel formulae:B Delta}) follows from the definition of $B_\Delta(n)$, and formula (\ref{prop:borel formulae:index B Delta}) follows from Lagrange's theorem. 
\end{proof}

\par
It will occasionally be useful to rewrite Proposition \enumref{prop:borel formulae}{prop:borel formulae:index B Delta} as
\[ [\GL_2(\ZZ/n\ZZ) : B_\Delta(n)] = \frac{\phi(n)}{\#\Delta} [\GL_2(\ZZ/n\ZZ) : B_0(n)] . \]

Let $n$ be a positive integer and let $\Delta$ be a subgroup of $(\ZZ/n\ZZ)^\times$ be given. We define the \emph{modular curve $X_\Delta(n)$} to be the modular curve $X_{B_\Delta(n)}$. If $\Delta$ is such that $\{\,\pm 1 \,\} \subsetneq \Delta \subsetneq (\ZZ/n\ZZ)^\times$, we call $X_\Delta(n)$ an \emph{intermediate modular curve}, since it is between $X_1(n)$ and $X_0(n)$. 
% If $m$ divides $n$, then there is a reduction homomorphism $\pi : (\ZZ/n\ZZ)^\times \to (\ZZ/m\ZZ)^\times$, which induces a natural map $X_\Delta(n) \to X_{\pi(\Delta)}(m)$. In this setting, we may denote the modular curve $X_{\pi(\Delta)}(m)$ by $X_\Delta(m)$, with the understanding that we reduce $\Delta$ modulo $m$. 
If $\Delta$ is a subgroup of $(\ZZ/n\ZZ)^\times$ and $m$ divides $n$, then we may write $X_\Delta(m)$ to mean the modular curve $X_{B_\Delta(n)}(m)$, with the understanding that we reduce $\Delta$ modulo $m$. 

\par
We briefly describe another construction of $X_\Delta(n)$.
% We may also construct $X_\Delta(n)$ in the following way: 
Define a congruence subgroup associated to $\Delta$ as follows:
\[\Gamma_\Delta(n) \defeq \set{\twomat{a}{b}{c}{d} \in \SL_2(\ZZ)}{c \equiv 0 \mod{n}\,,\ (a \mod{n}) \in \Delta}.\]
The group $\Gamma_\Delta(n)$ inherits from $\SL_2(\ZZ)$ the action on the upper half plane $\HH$ by linear fractional transformations. This action can be extended to ${\HH^* \defeq \HH \cup \PP^1(\QQ)}$, the extended upper half plane, and the quotient space $\HH^*/\Gamma_\Delta(n)$ is a Riemann surface. There exists a smooth projective curve $X(\Gamma_\Delta(n))/\QQ$ and a complex analytic isomorphism ${\HH^*/\Gamma_\Delta(n) \to X(\Gamma_\Delta(n))(\CC)}$ (c.f. \cite[Remark C.13.2]{Sil09}).
% \comment{add better reference. Derickx intro from some paper?} 
Moreover, the curves $X(\Gamma_\Delta(n))$ and $X_\Delta(n)$ are isomorphic over $\QQ$.
% \comment{I very much need a reference for this}
% 
% \par
Note that $\HH^*$ is fixed under the action by $-I \in \SL_2(\ZZ)$, so for any subgroup $\Delta < (\ZZ/n\ZZ)^\times$, the curves $X_{\Delta}(n)$ and $X_{\pm \Delta}(n)$ are isomorphic. 
% For this reason, to study intermediate modular curves $X_\Delta(n)$, we need only consider subgroups $\Delta$ containing $-1$. 
% For this reason, in this paper we will usually\footnote{See remark \ref{remark:-1 subtlety} for exceptions.} only consider subgroups $\Delta$ containing $-1$. 
For this reason, in this paper we only consider subgroups $\Delta$ containing $-1$. 

\par 
% \notice
% We now give a description of the geometric points of $X_\Delta(n)$. 
% Let $n \in \NN$ and $\{\, \pm 1 \,\} \leq \Delta \leq (\ZZ/n\ZZ)^\times$ be given. 
% Let $E/\QQ(j(E))$ be a non-CM elliptic curve with a point $P \in E$ of order $n$. 
For a point $P$ of order $n$ on an elliptic curve $E$, we 
define $\Delta P$ to be the set $\set{\delta P}{\delta \in \Delta}$ and $\QQ(\Delta P)$ to be the fixed field of $\set{\sigma \in \Gal_{\QQ}}{\sigma(\Delta P) = \Delta P}$. 
Then the non-cuspidal points of $X_\Delta(n)(\overline{\QQ})$ parametrize equivalence classes of pairs $(E, \Delta P)$, where $E$ is an elliptic curve and $P \in E$ is a point of order $n$. 
Moreover, with this construction, every non-cuspidal $k$-rational point of $X_\Delta(n)$ is of the form $[(E, \Delta P)]$, with $E/k$ an elliptic curve and $P \in E(k)$ a point of order $n$ such that $\Delta P$ is fixed by $\Gal_k$.

\begin{lemma} \label{lemma:deg(x) for x in X Delta}
    Let $n \in \NN$ and $\{\, \pm 1 \,\} \leq \Delta \leq (\ZZ/n\ZZ)^\times$ be given. Let $E/\QQ(j(E))$ be a non-CM elliptic curve and let $P \in E$ be a point of order $n$. 
    % Then the degree (over $\QQ$) of the closed point $x \in X_\Delta(n)$ associated to $(E, P)$ is \notice
    If $x \in X_\Delta(n)$ is the closed point associated to $(E,\Delta P)$, then 
    \[ \deg(x) = [\QQ(j(E), \Delta P):\QQ] . \]
    % Specifically, \notice
    % \[ \QQ(x) \simeq \QQ(j(E), \Delta P) . \]
\end{lemma}

\begin{proof}
    % By proposition \comment{include a reference---Poonen, 2.4.6}, we know that $[\QQ(x) : \QQ]$ is the size of the $\Gal_\QQ$-orbit of $[(E,\Delta P)]$. \comment{explain base change} 
    Let $\cS \leq \Gal_\QQ$ denote the stabilizer of the action of $\Gal_\QQ$ on $[(E,\Delta P)]$. 
    Since $E$ is defined over $\QQ(j(E))$, we have $\sigma E = E$ for all $\sigma \in \Gal_{\QQ(j(E))}$. Moreover, if $\sigma \in \Gal_\QQ$ fixes $E$, then there exists an isomorphism $\phi : E \to \sigma E$ defined over $\overline{\QQ}$. This implies $j(E) = j(\sigma E)$, so $\sigma \in \Gal_{\QQ(j(E))}$. 
    We then have 
    \begin{align*}
        \cS 
        &= \set{\sigma \in \Gal_\QQ}{\sigma[(E, \Delta P)] = [(E, \Delta P)]} 
        \\ &= \set{\sigma \in \Gal_\QQ}{\sigma E = E} \cap \set{\sigma \in \Gal_\QQ}{\sigma(\Delta P) = \Delta P} 
        \\ &= \Gal_{\QQ(j(E))} \cap \Gal_{(\QQ(\Delta P))} .
    \end{align*}
    By the orbit-stabilizer theorem and Galois theory, we conclude
    \[ [\QQ(x) : \QQ] = [\Gal_\QQ : \cS] = [\QQ^{\cS} : \QQ] = [\QQ(j(E))\QQ(\Delta P) : \QQ] = [\QQ(j(E), \Delta P):\QQ] . \qedhere \] 
\end{proof}

\section{Intermediate modular curves of prime-power level} \label{sec:IMC prime-power level}

\subsection{Rational isolated \texorpdfstring{$j$}{j}-invariants on \texorpdfstring{$X_\Delta(\ell^n)$}{X\_Delta(ell\string^n)}}

In this section, we consider the rational $j$-invariants that arise from isolated points on intermediate modular curves of prime-power level. 
Ejder \cite{Ejd22} proved that there are finitely many rational $j$-invariants arising from isolated points on $X_1(\ell^n)$, ranging over all primes $\ell > 7$, and gave a partial classification of which rational $j$-invariants may appear. 
Terao \cite[Theorem 1.5]{Ter24} showed that if $H \leq \GL_2(\hat{\ZZ})$ has level $7$ and $x \in X_H$ is an isolated point with $j(x)$ rational and non-CM, then $j(x) = 3^3 \cdot 5 \cdot 7^5 / 2^7$ and $H$ is conjugate to one of nine known subgroups.
Bourdon and Ejder \cite{Bou25-2} completed the classification begun by Ejder \cite{Ejd22} for $X_1(\ell^n)$ and proved an analogous result for the family of modular curves $X_0(\ell^n)$. 

\begin{theorem}[{\cite[Theorems 1 and 2.]{Bou25-2}}] \label{thm:Bou25-2 X_1 and X_0}
    Let $\ell$ be a prime and let $j \in \QQ$ be a rational number.
    \begin{enumerate}
        \item There exists an isolated point $x \in X_1(\ell^n)$ with $j(x) = j$ for some $n \in \NN$ if and only if $j$ is a CM $j$-invariant, $-7 \cdot 11^3$, or $-7 \cdot 137^3 \cdot 2083^3$. Moreover, the non-CM $j$-invariants occur if and only if $\ell = 37$.
        \item There exists an isolated point $x \in X_0(\ell^n)$ with $j(x) = j$ for some $n \in \NN$ if and only if $j$ is a CM $j$-invariant, $-11 \cdot 131^3$, $-11^2$, $-17^2 \cdot 101^3 / 2$, $-17 \cdot 373^3 / 2^{17}$, $-7 \cdot 11^3$, or $-7 \cdot 137^3 \cdot 2083^3$. Moreover, the non-CM $j$-invariants in this list correspond to isolated rational points on $X_0(\ell)$.
    \end{enumerate}
\end{theorem}

% \par 
% This is also related to work of Terao \cite[Theorem 1.5]{Ter24}, who showed that if $H \leq \GL_2(\hat{\ZZ})$ has level $7$, then the only non-CM rational $j$-invariant arising from an isolated point on $X_H$ is $3^3 \cdot 5 \cdot 7^5 / 2^7$. 

\par
Our goal for this section is to consider the above theorem in the setting of intermediate modular curves of prime-power level. 
Notice that in the above theorem, the non-CM rational $j$-invariants arising from isolated points on $X_1(\ell^n)$ all appear as $j$-invariants arising from isolated points on $X_0(\ell)$. We will see that any rational $j$-invariant corresponding to an isolated point on a modular curve $X_\Delta(\ell^n)$ must appear in the list of rational isolated $j$-invariants arising from $X_0(\ell)$. 

\subsection{Possible \texorpdfstring{$\ell$}{l}-adic images of Galois.}
We shall need a classification by Rouse, Sutherland, and Zureick-Brown \cite{RSZB22} of the possible $\ell$-adic images of Galois associated to non-CM elliptic curves over $\QQ$. The classification builds on work of Sutherland and Zywina \cite{Sut17} and of Rouse and Zureick-Brown \cite{Rou15}. 

\begin{theorem}[{\cite[Theorem 1.1.6]{RSZB22}}] \label{thm:RSZB l-adic} %, {\cite[Proposition 1]{Bou25-2}}] \label{thm:RSZB l-adic}
    Let $E/\QQ$ be a non-CM elliptic curve, let $\ell$ be a prime, and denote $R \defeq \rho_{E, \ell^\infty}(\Gal_\QQ)$. Then exactly one of the following is true:
    \begin{enumerate}[label=(\alph*)]
        \item The modular curve $X_R$ is isomorphic to $\PP^1$ or a rank one elliptic curve. \label{thm:RSZB:a}
        \item The modular curve $X_R$ has an exceptional rational point for known $R$. \label{thm:RSZB:b}
        \item $R$ is conjugate to a subgroup of $\Cns(3^3), \Cns(5^2), \Cns(7^2), \Cns(11^2)$ or $\Cns(\ell)$ for $\ell \geq 19$. \label{thm:RSZB:c}
        \item 
        % $R$ is contained in \verb|49.147.9.1| or \verb|49.196.9.1|. 
        $R$ is conjugate to a subgroup of \verb|49.147.9.1| or \verb|49.196.9.1|.
        \label{thm:RSZB:d}
    \end{enumerate}
\end{theorem}

For a non-CM $j$-invariant attached to an elliptic curve $E/\QQ$, Theorem \ref{thm:RSZB l-adic} gives us four cases to consider. 
If $R$ is as in case \ref{thm:RSZB:a}, the possible subgroups that can arise as $R$ are all known. Specifically, {\cite[Corollary 1.6]{Sut17}} says that for $\ell = 2, 3, 5, 7, 11, 13$, there are $1201, 47, 23, 15, 2, 11$ subgroups of $\GL_2(\ZZ_\ell)$ that can arise as $R$, respectively. And for $\ell > 13$, the only possible subgroup is $R = \GL_2(\ZZ_\ell)$. 
If $R$ is as in case \ref{thm:RSZB:b}, there are $23$ possible known images for $R$, which have levels $2^4, 2^3, 5^2, 7, 11, 13, 17, 37$. It is conjectured {\cite[Conjecture 1.1.5]{RSZB22}} that these known subgroups are all of the exceptional groups of prime-power level. 
If $R$ is as in case \ref{thm:RSZB:d} and is conjugate to a subgroup of \verb|49.196.1|, then Bourdon and Ejder \cite[Proposition 1]{Bou25-2} showed that $R$ must be conjugate to \verb|49.196.9.1|.

\begin{proposition}[{\cite[Proposition 3.1]{Ejd22}}, {\cite[Proposition 4]{Bou25-2}}] \label{prop:Bou25-2 4}
    Let $E/\QQ$ be a non-CM elliptic curve. Suppose $\rho_{E,\ell} (\Gal_\QQ)$ is conjugate to the normalizer of a nonsplit Cartan subgroup for some prime $\ell > 13$. For $P \in E(\overline{\QQ})$ of order $\ell^n$, let $x = [(E,P)] \in X_1(\ell^n)$ be the associated closed point on the modular curve. Then
    \[ \deg(x) = \frac{1}{2} \p{\ell^2 - 1} \ell^{2n-2} = \deg(X_1(\ell^n) \to X(1)) . \]
\end{proposition}

As noted in \cite{Bou25-2}, there was in an error in a result used to prove \cite[Proposition 3.1]{Ejd22}, so the proof is given in \cite[Proposition 4]{Bou25-2}.
From the above proposition, we immediately obtain the following corollary.

\begin{corollary} \label{cor:used for case c}
    Let $E/\QQ$ be a non-CM elliptic curve. Suppose $\rho_{E,\ell}(\Gal_\QQ)$ is conjugate to the normalizer of a nonsplit Cartan subgroup for some prime $\ell > 13$. For $P \in E(\overline{\QQ})$ of order $\ell^n$ and $\{\, \pm 1 \,\} \leq \Delta \leq (\ZZ/\ell^n \ZZ)^\times$, let $x = [(E,\Delta P)] \in X_\Delta(\ell^n)$ be the associated closed point on the modular curve. Then
    \[ \deg(x) = \frac{1}{\#\Delta} \p{\ell^2 - 1} \ell^{2n-2}  = \deg\p{X_\Delta(\ell^n) \to X(1)} . \]
\end{corollary}

We will make use of the following theorem.

\begin{theorem}[{\cite[Theorem 1.3]{Ter24}}] \label{thm:Ter24 genus 0}
    Let $n \geq 1$, let $H \leq \GL_2(\ZZ /n \ZZ)$ be a subgroup, and let $x$ be a non-cuspidal isolated point on $X_H$ with $j(x) \neq \{\,0, 1728\,\}$. Let $E/\QQ$ be an elliptic curve such that $j(E) = j(x)$, and let $G_n \defeq \rho_{E, n}(\Gal_{\QQ(j(x))})$. Then the modular curve $X_{G_n}$ contains an isolated point with $j$-invariant equal to $j(x)$.
\end{theorem}

If the modular curve $X_{G_n}$ in the above theorem has genus $0$, then Theorem \ref{thm:RR isolated} implies that there are no isolated points on $X_{G_n}$ and hence $x$ cannot be isolated. 
This will allow us to rule out certain cases in the following proposition. 

\begin{proposition} \label{prop:X Delta prime power:cases a and b}
    Let $\ell\in \NN$ be prime and $n \in \NN$ a positive integer. Let $E/\QQ$ be a non-CM elliptic curve and suppose $\rho_{E,\ell^\infty}(\Gal_\QQ)$ is conjugate to a known image, as in cases \ref{thm:RSZB:a} or \ref{thm:RSZB:b} of Theorem \ref{thm:RSZB l-adic}. If $\ell \neq 11, 17, 37$, then there are no isolated points above $j(E)$ on a modular curve $X_\Delta(\ell^n)$. 
\end{proposition}

\begin{proof}
    First, suppose $R \defeq \rho_{E, \ell^\infty}(\Gal_\QQ)$ is equal to $\GL_2(\ZZ_\ell)$, for some prime $\ell$. Then for every $n \in \NN$ and $\Delta \leq (\ZZ/ \ell^n \ZZ)^\times$, the fiber of $X_\Delta(\ell^n) \to X(1)$ over $j(E)$ is a single closed point of maximal degree. So, if there is an isolated point on $X_\Delta(\ell^n)$ above $j(E)$, then Theorem \ref{thm:Bou19 4.3} implies $j(E) \in X(1)$ is isolated. But $X(1)$ has genus $0$, so by Theorem \ref{thm:RR isolated} this is not possible. 
    
    We now assume $R$ is conjugate to a subgroup not equal to $\GL_2(\ZZ_\ell)$ that appears in \cite[Tables 1-4]{Sut17} or in \cite[Table 1]{RSZB22}; these are cases \ref{thm:RSZB:a} and \ref{thm:RSZB:b} of Theorem \ref{thm:RSZB l-adic}.
    % When run on all of these known images, the output of the $\Gamma_1$-algorithm returns empty except for the subgroups \verb|17.72.1.2|, \verb|37.114.4.1|, and \verb|37.114.4.2|. 
    By Theorem \ref{thm:Bou25-2 X_1 and X_0}, there are no isolated points on $X_0(\ell^n)$ or $X_1(\ell^n)$ associated to $E$ because $\ell \neq 11,17,37$, so it suffices to consider intermediate modular curves $X_\Delta(\ell^n)$.
    So, let $n$ be a positive integer and $\{\, \pm 1 \,\} \subsetneq \Delta \subsetneq (\ZZ/\ell^n\ZZ)^\times$ a subgroup. 
    Suppose $x \in X_\Delta(\ell^n)$ is an isolated point with $j(x) = j(E)$ and say $x = [(E, \Delta P)]$.
    % There exists a $B_{\Delta}(\ell^n)$-level structure $[\alpha]_{B_{\Delta}(\ell^n)}$ such that $(E, [\alpha]_{B_{\Delta}(\ell^n)})$ is a minimal representative for $x$. 
    % There exists a point $P \in E(\overline{\QQ})$ such that $(E, \Delta P)$ is a representative for $x$. 
    By Theorem \ref{thm:RSZB l-adic}, and because we are assuming $\ell \neq 11, 17, 37$, we know $R$ has level $\ell^m$ for some 
    \[ \ell^m \in \cL \defeq \{\, 1, 2, 4, 8, 16, 32, 3, 9, 27, 5, 25, 7, 13 \,\}. \] 
    If $n \geq m$, then by Proposition \ref{prop:fixed prime power deg} and Theorem \ref{thm:Bou19 4.3}, the image of $x$ under the natural map $X_\Delta(\ell^n) \to X_\Delta(\ell^m)$ is isolated. 
    So, it suffices to consider the case when $\ell^n \in \cL$. 
    Note there are no proper subgroups $\{\, \pm 1 \,\} \subsetneq \Delta \subsetneq (\ZZ/\ell^n\ZZ)^\times$ for 
    \[ \ell^n \in \{\, 1, 2, 4, 8, 3, 9, 5, 7 \,\} . \]
    % so for each of these levels the result follows by Theorem \ref{thm:Bou25-2 X_1 and X_0}. 
    Moreover, every intermediate modular curve of level $13$ and $16$ has genus zero, and hence has no isolated points by Theorem \ref{thm:RR isolated}. 
    
    So, it remains to consider the case when $\ell^n \in \{\, 32, 27, 25 \,\}$. 
    There are $132$ possible choices for $R$ to consider, and these can be found 
    % in lines 1182-1313 
    in the file \href{https://github.com/abbey-bourdon/rational-isolated-prime-power/blob/main/elladicgens.txt}{\texttt{elladicgens.txt}} associated to \cite{Bou25-2}. 
    Let $f$ denote the natural map $X_1(\ell^n) \to X_\Delta(\ell^n)$ and let $x' \in X_1(\ell^n)$ be a closed point in the fiber of $f$ over $x$.
    % Let $x' \in X_1(\ell^n)$ be a closed point in the fiber of $f : X_1(\ell^n) \to X_\Delta(\ell^n)$ over $x$. 
    % As explained in the proof of \cite[Theorem 4]{Bou25-2}, when run on the known $\ell$-adic images associated to non-CM elliptic curves $E'/\QQ$, as appearing in \cite[Tables 1-4]{Sut17} and \cite[Table 1]{RSZB22}, the output of the $\Gamma_1$-algorithm is the empty set unless $\rho_{E', \ell^\infty} =$ \verb|17.72.1.2|, \verb|37.114.4.1|, or \verb|37.114.4.2|. 
    % We may therefore assume that output of the $\Gamma_1$-algorithm is empty when run on $R$. 
    We use \verb|Magma| to find that all but $28$ of the $132$ images have the property that every closed point in the fiber of $f$ over $x$ has degree strictly greater than $(\#\Delta / 2) \genus(X_\Delta(\ell^n))$. The data in Table \ref{tab:X Delta ell n data} gives the exact bound on $\deg(x')$ for each level $\ell^n$. 
    If $\deg(x') > (\#\Delta / 2) \genus(X_\Delta(\ell^n))$, then Lemma \ref{lemma:closed point deg leq} implies 
    \[ \deg(x) \geq \frac{2 \deg(x')}{\#\Delta} > \frac{2 \genus\p{X_1(\ell^n)}}{\#\Delta} > \genus\p{X_\Delta(\ell^n)}, \]
    which contradicts that $x$ is isolated by Theorem \ref{thm:RR isolated}.
    So, we may assume that $R$ is one of the remaining $28$ images. A \verb|Magma| computation shows that the modular curve $X_R(\ell)$ has genus zero, so Theorem \ref{thm:RR isolated} implies there are no isolated points on $X_R(\ell)$. But this contradicts Theorem \ref{thm:Ter24 genus 0}, so we conclude that $x$ cannot be isolated. 
    See the website of the author for the \verb|Magma| code used. 
\end{proof}

\begin{table}[htp]
    \centering
    \begin{tabular}{c c c c}
        \hline
        $\ell^n$ & $\#\Delta$ & $\genus(X_\Delta(\ell^n))$ & $\frac{\# \Delta}{2} \genus(X_\Delta(\ell^n))$
        \\ \hline \hline
        $25$ & $4$ & $4$ & $8$
        \\ $25$ & $10$ & $0$ & $0$
        \\ \hline
        $27$ & $6$ & $1$ & $3$
        \\ \hline
        $32$ & $4$ & $5$ & $10$
        \\ $32$ & $8 $& $1$ & $4$
        \\ \hline
    \end{tabular}
    \caption{Data for intermediate modular curves $X_\Delta(\ell^n)$.}
    \label{tab:X Delta ell n data}
\end{table}

\par 
We now consider the situation when $R \defeq \rho_{E, \ell^\infty} (\Gal_\QQ)$ is as in case \ref{thm:RSZB:c} of Theorem \ref{thm:RSZB l-adic}. That is, we have $R \leq \Cns(3^3), \Cns(5^2), \Cns(7^2), \Cns(11^2)$ or $\Cns(\ell)$ for $\ell \geq 19$. For this, a key ingredient is work of Furio \cite[Theorem 1.9]{Fur25}, which characterizes the possibilities for $R$ in this case.

\begin{proposition} \label{prop:X Delta prime power:case c}
    Let $\ell$ be an odd prime and $n \in \NN$. Let $E/\QQ$ be a non-CM elliptic curve and $\Delta$ a subgroup $\{\, \pm 1 \,\} < \Delta < (\ZZ/\ell^n\ZZ)^\times$. If $\rho_{E, \ell} (\Gal_\QQ)$ is conjugate to a subgroup of the normalizer of a non-split Cartan subgroup, then there is no isolated point $x \in X_\Delta(\ell^n)$ with $j(x) = j(E)$. 
\end{proposition}

\begin{proof}
    Let $x \in X_\Delta(\ell^n)$ be such that $j(x) = j(E)$ and say $x = [(E, \Delta P)]$. 
    By Proposition \ref{prop:X Delta prime power:cases a and b}, we may assume $R \defeq \rho_{E, \ell^\infty} (\Gal_\QQ)$ is not a known image. By \cite[Theorem 1.9]{Fur25}, either $R$ has level $\ell^d$ and is conjugate to $\Cns(\ell^d)$ or $R$ has level $\ell^2$ and 
    \[ R(\ell^2) \simeq \Cns(\ell) \ltimes \left\{\, I + \ell \twomat{a}{\epsilon b}{-b}{c} \,\right\} , \]
    with the semidirect product defined by the conjugation action. 
    
    \par Suppose $R$ has level $\ell^d$ and is conjugate to $\Cns(\ell^d)$. The proof of \cite[Theorem 6]{Bou25-2}, which proves the claim for the case of $X_\Delta(\ell^d) = X_1(\ell^d)$, shows that $[(E,P)] \in X_1(\ell^n)$ has degree $\ell^{2n-2}(\ell^2 - 1)/2$, so $\deg(x) = \ell^{2n-1}(\ell^2 - 1)/(\# \Delta)$. But then Theorem \ref{thm:Bou19 4.3} implies that the image of $x$ under the natural map $X_\Delta(\ell^n) \to X(1)$ is isolated, which is absurd, since $X(1)$ has genus $0$. 
    
    \par Now suppose $R$ has level $\ell^2$ and is isomorphic to the semidirect product above. 
    The proof of \cite[Theorem 6]{Bou25-2} shows that $[(E,P)] \in X_1(\ell^2)$ has degree at least $\ell(\ell^2 - 1)/2$, so we have $\deg(x) \geq \ell(\ell^2 - 1)/\#\Delta$. 
    If $\ell > 13$, then Corollary \ref{cor:used for case c} and Theorem \ref{thm:Bou19 4.3} imply there are no isolated points on $X_\Delta(\ell^n)$. 
    
    If $3 \leq \ell \leq 11$, then there are no intermediate modular curves of level $\ell$ because $\phi(\ell)/2$ is prime. There is one intermediate modular curve of level $13$, 
    % the curve $X_\Delta(13)$ for $\#\Delta = 4$, 
    but this curve has genus $0$ and hence has no isolated points by Theorem \ref{thm:RR isolated}. 
    Hence, we may assume $3 \leq \ell \leq 13$ and $n \geq 2$. 
    % If $n \geq 2$, then 
    Proposition \ref{prop:fixed prime power deg} and Theorem \ref{thm:Bou19 4.3} imply that if $x$ is isolated, then the image of $x$ under the natural map $f : X_\Delta(\ell^n) \to X_\Delta(\ell^2)$ is isolated. Note that there are no intermediate modular curves of level $9$ because $\phi(9)/2 = 3$ is prime, so we may assume $\ell \geq 5$. The data in Table \ref{tab:X Delta ell 2 data} combined with Theorem \ref{thm:RR isolated} shows that $f(x)$, and hence $x$, is not isolated. \qedhere 
\end{proof}

\begin{table}[htp]
    \centering
    \begin{tabular}{c c c c}
        \hline
        $\ell$ & $\#\Delta$ & $\operatorname{genus}(X_\Delta(\ell^2))$ & $\ell(\ell^2 - 1)/\# \Delta$
        \\ \hline\hline
        $5$ & 4 & $4$ & $30$
        \\ $5$ & 10 & $0$ & $12$
        \\ \hline
        $7$ & 6 & $19$ & $56$
        \\ $7$ & 14 & $3$ & $24$
        \\ \hline
        $11$ & 10 & $106$ & $132$
        \\ $11$ & 22 & $26$ & $60$
        \\ \hline
        $13$ & 4 & $516$ & $546$
        \\ $13$ & 6 & $340$ & $364$
        \\ $13$ & 12 & $164$ & $182$
        \\ $13$ & 26 & $50$ & $84$
        \\ $13$ & 52 & $24$ & $42$
        \\ $13$ & 78 & $16$ & $28$
        \\ \hline
    \end{tabular}
    \caption{Data for $X_\Delta(\ell^2)$ with $\ell \leq 13$.}
    \label{tab:X Delta ell 2 data}
\end{table}

% If $E/\QQ$ is a non-CM elliptic curve with $\rho_{E,7^\infty}(\Gal_\QQ)$ conjugate to a subgroup of \verb|49.147.9.1|, then $\rho_{E,7}(\Gal_\QQ)$ is contained in the normalizer of a non-split Cartan subgroup.\notice 
% So, by 
By 
% Proposition \ref{prop:X Delta prime power:case c} 
the previous proposition
and \cite[Proposition 1]{Bou25-2}, it remains to consider the case when $\rho_{E, \ell^\infty}(\Gal_\QQ)$ is conjugate to \verb|49.196.9.1|. 

\begin{proposition} \label{prop:X Delta prime power:case d}
    Let $E/\QQ$ be a non-CM elliptic curve and suppose $\rho_{E, 7^\infty}(\Gal_\QQ)$ is conjugate to \verb|49.196.9.1|. Then there are no isolated points above $j(E)$ on a modular curve $X_\Delta(7^n)$. 
\end{proposition}

\begin{proof}
    Let $n$ be a positive integer and let $E/\QQ$ be a non-CM elliptic curve with $\rho_{E, 7^\infty}(\Gal_\QQ)$ conjugate to \verb|49.196.9.1|. Suppose $x \in X_\Delta(7^n)$ is an isolated point with $j(x) = j(E)$. 
    If $n = 1$, then $X_\Delta(7^n)$ is not an intermediate modular curve because $\phi(7)/2 = 3$ is prime, so Theorem \ref{thm:Bou25-2 X_1 and X_0} implies $x$ cannot be isolated. If $n \geq 2$, then by Proposition \ref{prop:fixed prime power deg} and Theorem \ref{thm:Bou19 4.3} imply that the image of $x$ under the natural map to $X_\Delta(49)$ is isolated. So, it suffices to show that $[(E, \Delta P)] \in X_\Delta(49)$ is not isolated. 
    
    A \verb|Magma| computation shows that if $x = [(E,\langle P \rangle )] \in X_0(49)$ is a closed point with \break $j(x) = j(E)$, then $\deg(x) = 56$. 
    % For each intermediate modular curve $X_\Delta(49)$, we have $\genus(X_{\Delta}(49)) < 56$, 
    The genus of each intermediate modular curve $X_\Delta(49)$ is strictly less than $56$,
    so $[(E, \Delta P)] \in X_\Delta(49)$ is not isolated by Theorem \ref{thm:RR isolated}. 
\end{proof}

\subsection{Proof of Theorem \ref{thm:intro:X Delta prime power}.}

We restate Theorem \ref{thm:intro:X Delta prime power} below for convenience. 

\begin{theorem} \label{thm:X Delta prime power}
    Let $\ell \in \NN$ be a prime, let $n \in \NN$ a positive number, and suppose $\Delta \leq (\ZZ/\ell^n \ZZ)^\times$ is a subgroup containing $-1$. 
    If $x \in X_\Delta(\ell^n)$ is a non-cuspidal isolated point with non-CM rational $j$-invariant $j \in \QQ$, then $j$ appears as the $j$-invariant of an isolated point on $X_0(\ell)$, for some $\ell \in \{\, 11, 17, 37 \,\}$. 
\end{theorem}

% \begin{theorem} \label{thm:X Delta prime power}
%     Let $\ell \in \NN$ be a prime, let $n \in \NN$ a positive number, and suppose $\Delta \leq (\ZZ/\ell^n \ZZ)^\times$ is a subgroup containing $-1$. 
%     If $x \in X_\Delta(\ell^n)$ is a non-cuspidal isolated point with non-CM rational $j$-invariant $j \in \QQ$, then $j$ appears as the $j$-invariant of an isolated point on $X_0(\ell)$, for some $\ell \in \{\, 11, 17, 37 \,\}$. 
% \end{theorem}

\begin{proof}
    % \comment{CM $j$-invariants: Use Clark ``Volcanoes... I'' Theorem 1.2 for $j$ away from $0, 1728$.} 
    %
    % Suppose $j \in \QQ$ is a CM $j$-invariant and $x = [(E, \Delta P)] \in X_{\Delta}(\ell^n)$ is an isolated point with $j(x) = j$. By Theorem \ref{thm:Bou25-2 X_1 and X_0}, $[(E, P)] \in X_1(\ell^n)$ is isolated. 
    % % Let $f$ be the natural map $X_1(\ell^n) \to X_0(\ell^n)$. 
    % If $j \neq 0, 1728$, then \cite[Theorem 1.2]{Cla22} implies 
    % \[ \deg([(E, P)]) = \deg\p{X_1(\ell^n) \to X_0(\ell^n)} \deg ([(E, \langle P \rangle)]) ,  \]
    % which implies 
    % \[ \deg([(E, P)]) = \deg\p{X_1(\ell^n) \to X_\Delta(\ell^n)} \deg ([(E, \Delta P )]) . \]
    % Then $[(E, \Delta P)] \in X_{\Delta}(\ell^n)$ is isolated by Theorem \ref{thm:Bou19 4.3}. 
    % \comment{What if $j = 0, 1728$?} 
    
    \par
    Let $E/\QQ$ be a non-CM elliptic curve and let $\ell \in \NN$ be prime. 
    % Denote $R$ to be the image $\rho_{E, \ell^\infty}(\Gal_\QQ) \leq \GL_2(\ZZ_\ell)$. 
    By Theorem \ref{thm:RSZB l-adic}, there are four cases for the image $R \defeq \rho_{E, \ell^\infty}(\Gal_\QQ) \leq \GL_2(\ZZ_\ell)$. If $R$ is in case \ref{thm:RSZB:a} or \ref{thm:RSZB:b} and $\ell \neq 11,17, 37$, then Proposition \ref{prop:X Delta prime power:cases a and b} implies that $j(E)$ is not the $j$-invariant of an isolated point on a modular curve $X_\Delta(\ell^n)$. Similarly, if $R$ is in case \ref{thm:RSZB:c} or case \ref{thm:RSZB:d}, then Proposition \ref{prop:X Delta prime power:case c} and Proposition \ref{prop:X Delta prime power:case d} imply that $j(E)$ does not appear as the $j$-invariant of an isolated point on a modular curve $X_\Delta(\ell^n)$. 
    
    \par
    It remains to consider the case when $\ell = 11, 17, 37$ and $R$ is a known image. Because $\phi(11)/2$ is prime, there are no intermediate modular curves of level $11$. The $j$-invariants of the known images of levels $17$ and $37$ all arise from rational points on $X_0(17)$ or $X_0(37)$, as desired. 
\end{proof}

\DeclareEmphSequence{\textrm\itshape}
\printbibliography[heading=bibnumbered]
\renewcommand{\footnotesize}{\normalsize}

\end{document}